\documentclass{amsart}
\usepackage{amsmath}
\usepackage{amsthm}
\usepackage{amsfonts}
\usepackage{amssymb}
\usepackage[all]{xy}
\usepackage{alltt}
\usepackage{multicol}
\usepackage{pdftexcmds}
\usepackage{etoolbox}
\usepackage{ifthen}
\usepackage[
        colorlinks=true,urlcolor=blue,linkcolor=blue,citecolor=blue
]{hyperref}
\usepackage{textcomp}
\usepackage{fonttable}
\usepackage{enumitem}
\usepackage{graphicx}
\usepackage{comment}
\usepackage{tikz}                     
\usepackage{centernot}

\theoremstyle{plain}
\newtheorem{theorem}{Theorem}[section]

\newtheorem{corollary}[theorem]{Corollary}

\newtheorem{proposition}[theorem]{Proposition}

\theoremstyle{definition}
\newtheorem{ex}[theorem]{Example}

\theoremstyle{definition}

\newtheorem{remark}[theorem]{Remark}
\newtheorem{question}[theorem]{Question}

\usepackage[margin=1.25in]{geometry}

\newcommand*\circled[1]{\tikz[baseline=(char.base)]{
            \node[shape=circle,draw,inner sep=1.5pt] (char) {#1};}}

\DeclareMathOperator{\drank}{d}

\newcommand{\Z}{\ensuremath{\mathbf{Z}}}
\newcommand{\F}{\ensuremath{\ff}}
\newcommand{\R}{\ensuremath{\mathbf{R}}}
\newcommand{\Q}{\ensuremath{\mathbf{Q}}}

\newcommand{\C}{\ensuremath{\mathbf{C}}}
\newcommand{\Dih}{\ensuremath{\mathbf{D}}}

\newcommand{\Quat}{\ensuremath{\mathbf{Q}}}
\newcommand{\U}{\ensuremath{\mathbf{U}}}

\newcommand{\A}{\ensuremath{\mathbf{A}}}

\newcommand{\n}{\noindent}
\newcommand{\map}[3]{\ensuremath{#1 : #2 \longrightarrow #3}}

\newcommand{\cyclic}[1]{\ensuremath{\langle #1 \rangle}}

\renewcommand{\F}{\ensuremath{\mathbf{F}}}

\newcommand{\trank}{\ensuremath{\mathrm{d}}}

\newcommand{\mm}[4]{\ensuremath{\left[ \begin{array}{cc} #1 & #2 \\ #3
        & #4\end{array} \right]}}

\begin{document}

\title{Fuchs' problem for $p$-groups}
\date{\today}
 
\author{Sunil K. Chebolu}
\address{Department of Mathematics \\
Illinois State University \\
Normal, IL 61790, USA}
\email{schebol@ilstu.edu}

\author{Keir Lockridge} 
\address {Department of Mathematics \\
Gettysburg College \\
Gettysburg, PA 17325, USA}
\email{klockrid@gettysburg.edu}

\thanks{The first author is supported by Simons Collaboration Grants for Mathematicians (516354). }

 
\begin{abstract}
Which groups can be the group of units in a ring? This open question, posed by L\'{a}szl\'{o} Fuchs in 1960, has been studied by the authors and others with a variety of restrictions on either the class of groups or the class of rings under consideration. In the present work, we investigate Fuchs' problem for the class of $p$-groups. Ditor provided a solution in the finite, odd-primary case in 1970. Our first main result is that a finite $2$-group $G$ is the group of units of a ring of odd characteristic if and only if $G$ is of the form $\C_8^t \times \prod_{i = 1}^k \C_{2^{n_i}}^{s_i},$ where $t$ and $s_i$ are non-negative integers and $2^{n_i}+1$ is a Fermat prime for all $i$. We also determine the finite abelian $2$-groups of rank at most 2 that are realizable over the class of rings of characteristic 2, and we give some results concerning the realizability of $2$-groups in characteristic 0 and $2^n$. Finally, we show that the only almost cyclic $2$-groups  which appear as the group of units in a ring are $\C_2, \C_4, \C_8, \C_{q-1}$ ($q$ a Fermat prime),  $\C_2 \times \C_{2^n} (n \ge 1)$, $\Dih_8$, and $\Quat_8$. From this list we obtain the $p$-groups with periodic cohomology which arise as the group of units in a ring.
\end{abstract}
 
\maketitle
\thispagestyle{empty}


\section{Introduction}\label{sec:introduction}

L\'{a}szl\'{o} Fuchs poses the following problem in \cite{fuchsprob}: determine which abelian groups are the group of units in a commutative ring. In \cite{PS}, following \cite{gilmer}, all finite cyclic groups which occur as the group of units of a ring are determined. In \cite{cl-3}, we provide an answer to this question for indecomposable abelian groups. Rather than narrowing the class of abelian groups under consideration, one could also broaden the scope of the problem by considering nonabelian groups and noncommutative rings. Call a group $G$ {\bf realizable} if it is the group of units in some ring, and call $G$ {\bf realizable in characteristic $n$} if it is the group of units in some ring of characteristic $n$. In \cite{do1} and \cite{do2} the authors determine which alternating, symmetric, and finite simple groups are realizable.  In \cite{cl-4} we determine the dihedral groups which are realizable. In \cite{dd}, finite groups of units in integral domains and torsion free rings are classified, and partial information is obtained about rings of characteristic zero.

In the present paper, we consider Fuchs' problem for $p$-groups (where $p$ is a prime). Ditor proved in \cite{ditor} that a group of odd order is realizable as the group of units in a ring if and only if $G$ is a direct product of cyclic groups whose orders are one less than a power of 2. Consequently, for odd primes $p$, a finite $p$-group is realizable if and only if it is an elementary abelian $p$-group and $p$ is a Mersenne prime. In \cite{monthly} we give an elementary proof of this fact over commutative rings.

Turning our attention to $2$-groups, where the problem is open and considerably more difficult,  in \S \ref{oddcharsec} we determine the possible characteristics of a ring whose group of units is a $2$-group and then solve Fuchs' problem for finite $2$-groups over rings of odd characteristic.

\begin{theorem}\label{oddchar}
Let $n$ be a positive integer. The following are equivalent:
\begin{enumerate}
\item There exists a $2$-group that is realizable in characteristic $n$. (There is a ring $R$ of characteristic $n$ such that $R^\times$ is a $2$-group.)
\item The integer $n$ is of the form $2^ap_1p_2 \cdots p_k,$ where the $p_i$'s are distinct Fermat primes, $a \ge 0$ and  $k \ge 0$. \label{nform}
\item It is possible to construct a regular $n$-gon with straightedge and compass. \label{gonitem}
\end{enumerate}
Moreover, the realizable finite $2$-groups over rings of odd characteristic are exactly the groups of the form 
\[ \C_8^t \times \prod_{i = 1}^k \C_{2^{n_i}}^{s_i}, \]
where $t$ and $s_i$ are non-negative integers and $2^{n_i}+1$ is a Fermat prime for all $i$.
\end{theorem}

\n The equivalence of (\ref{gonitem}) and (\ref{nform}) is the Gauss-Wantzel Theorem. The group $\C_n$ above denotes the (multiplicative) cyclic group of order $n$.

For finite $2$-groups, it remains to consider rings of characteristic  $2^n$ and $0$. In \S \ref{char2}, we consider Fuchs' problem for finite abelian $2$-groups and rings of characteristic 2. We provide several general results that allow us to prove that certain families of finite abelian $2$-groups are not realizable in characteristic 2. For a finite abelian group $G$, let $\drank(G)$ denote the number of factors in the canonical decomposition of $G$ as a product of cyclic groups of prime power order. We will refer to this quantity as the {\bf rank} of $G$. The following theorem follows from our work in \S \ref{char2}.

\begin{theorem} \label{rank1-2}
For any positive integer $n$, there are only finitely many finite abelian $2$-groups of rank $n$ that are realizable in characteristic $2$. The finite abelian $2$-groups of rank at most $2$ that are realizable in characteristic $2$ are $\C_8 \times \C_2$, $\C_4 \times \C_4$, $\C_4 \times \C_2, \C_2 \times \C_2, \C_4$, $\C_2$ and $\C_1$.
\end{theorem}

In \S \ref{char2n}, we give a brief account of what we know about the realizability of finite abelian $2$-groups in characteristic $2^n$ when $n > 1$. 

In \S \ref{char0sec}, we study Fuchs' problem for $2$-groups in characteristic 0.  The first examples of realizable $2$-groups in characteristic $0$ that come to mind are $\C_2 = \Z^\times$ and $\C_4 = \Z[i]^\times$. It turns out that these two groups are also important in the general case. If $B$ is a finite abelian $2$-group, write $B_2 = \{x^2 \, | \, x \in B\}$.

\begin{theorem}\label{char0main} If $G$ is a finite $2$-group that is realizable in characteristic $0$, then either $\C_2$ or $\C_4$ must be a summand of $G$. Further, the finite abelian $2$-groups of the form \[ \C_2 \times A \text{\,\,\, and\,\,\,\, } \C_4 \times A \times A \times B \times B_2\] are realizable in characteristic zero for any finite abelian $2$-groups $A$ and $B$ (either group may be trivial).
\end{theorem}

In \S \ref{nonabelian}, we consider nonabelian $2$-groups. We show that if a nonabelian $2$-group is realizable, then its center and all of its cyclic maximal abelian subgroups must be isomorphic to $\C_2$ or $\C_4$. We apply this result to solve Fuchs' problem over arbitrary rings for the classes of almost cyclic $2$-groups and groups with periodic cohomology. We will say that a $p$-group $G$ is {\bf almost cyclic} if $G$ has a cyclic subgroup of index $p$. The almost cyclic $2$-groups have been classified (see \cite[12.5.1]{hall}); every such group falls into one of five families (abelian, dihedral, quaternion, semidihedral, and modular). It is well known that a  $p$-group $G$ has periodic mod $p$ cohomology if and only if it is a cyclic $p$-group or a generalized quaternion $2$-group. Thus, groups with periodic cohomology are contained within the class of almost cyclic $p$-groups.  Building upon our earlier work in \cite{cl-3} and \cite{cl-4}, we prove:

\begin{theorem}\label{maintheorem}
Let $p$ be a prime and let $G$ be a realizable finite $p$-group.
\begin{enumerate}
 \item If $G$ is almost cyclic, then $G$ is either  $\C_2, \C_4, \C_8, \C_{q-1}$ ($q$ a Fermat prime),  $\C_2 \times \C_{2^n}$  $(n \ge 1)$,  $\C_p$ ($p$ Mersenne), $\C_p \times \C_p$ ($p$ Mersenne), $\Dih_8$, or $\Quat_8$. \label{maintheorem-ac}
 \item  If $G$ has periodic cohomology, then $G$ is either $\C_2, \C_4, \C_8, \C_p$ ($p$ Mersenne), $\C_{q-1}$ ($q$ Fermat) or $\Quat_8$. \label{maintheorem-pc}
 \end{enumerate}
\end{theorem}

\n The groups $\Dih_8$ and $\Q_8$ denote the dihedral group of order 8 and the quaternion group of order 8, respectively. 

We wish to thank the anonymous referee for bringing \cite{ddx} to our attention; this prompted our inclusion of \S \ref{char2n}.

\section{Preliminaries}

Let $p$ be a prime. We now collect a few results for later use. The next proposition will allow us to restrict our attention to finite rings when considering the realizability of finite $2$-groups, and we will be able to further restrict our attention to finite commutative rings when considering the realizability of finite abelian $2$-groups.

\begin{proposition}[\cite{cl-4}] \label{wlogRfinite}
If a finite group $G$ is realizable in characteristic $n > 0$, then there exists a finite ring $R$ of characteristic $n$ whose group of units is $G$. If $G$ is also abelian, then the ring $R$ may be taken to be a commutative ring.
\end{proposition}

The following propositions will enable us to make use of what is already known about the realizability of cyclic $p$-groups.

\begin{proposition} \label{maximal}
If a group $G$ is  realizable in characteristic $n$, then so are all of its maximal abelian subgroups (in characteristic $n$).
\end{proposition}
\begin{proof}
Suppose $R$ is a ring of characteristic $n$, and let $H$ be a maximal abelian subgroup of $G = R^\times$. Consider the subring $S$ of $R$ that is generated by the elements of $H$ over $\Z_n \subseteq R$. The ring $S$ is commutative with characteristic $n$, and $S^\times$ is an abelian subgroup of $R^\times = G$ which contains $H$. Since $H$ is a maximal abelian subgroup of $G$, we conclude that $S^\times = H$.
\end{proof}

\begin{proposition} \label{center}
If a group $G$ is  realizable in characteristic $n$, then so is its center (in characteristic $n$).
\end{proposition}
\begin{proof}
Suppose $R$ is a ring of characteristic $n$, and let $Z$ be the center of $G = R^\times$. Consider the subring $S$ of $R$ that is generated by the elements of $Z$ over $\Z_n \subseteq R$.  It is clear that $Z \subseteq S^\times$. For the other inclusion, note that any element of $S^\times$ is a sum of products of elements of $Z$. Such an element clearly commutes with every element of $G$, so it must be central. This proves that $S^\times \subseteq Z$, and the proof is complete. 
\end{proof}

\begin{corollary}
Let $R$ be a ring of characteristic $c$ and suppose $G = R^\times$ is a finite group whose maximal abelian subgroup $H$ is a cyclic $p$-group. Then, $c = 0, 2, 4, q$ or $2q$ where $q$ is a Fermat prime. Moreover, the possible values of $H$, given $c$, are:
\begin{center}
\begin{tabular}{|l|l|}\hline
$c = 0, 4$ & $H = \C_2, \C_4$ \\ \hline
$c = 2$ & $H = \C_2, \C_4, \C_p$ ($p$ a Mersenne prime) \\ \hline
$c = 3, 6$ & $H = \C_2, \C_8$ \\ \hline
$c = q, 2q$ ($q > 3$ a Fermat prime) & $H = \C_{q-1}$ \\ \hline
\end{tabular}
\end{center}
\end{corollary}

\begin{proof}
Let $G$ be as given. Then by Proposition \ref{maximal}, $H$, which is a cyclic $p$-group, is also realizable in the same characteristic as $R$. The result now follows from our work on the realizability of indecomposable abelian groups in \cite{cl-3}.
\end{proof}

We will make use of the following fact from ring theory (cf. \cite[3.1, 3.2]{ddx}).

\begin{proposition}\label{J0}
If $R$ is a finite ring of prime characteristic $p$ such that $\gcd(|R^\times|, p) = 1$, then \[ R \cong \prod_{i = 1}^n \F_{p^{k_i}}\] and \[ R^\times \cong \prod_{i = 1}^n \F_{p^{k_i}}^\times \cong \prod_{i = 1}^n \C_{p^{k_i} - 1}.\]\end{proposition}
\begin{proof}
Let $J$ denote the Jacobson radical of $R$. It is well known that $1 + J \leq R^\times$. This means that $|J|$ is both a power of $p$ (since $J$ is a two-sided ideal of $R$, which is an $\F_p$-vector space) and relatively prime to $p$ (since $|J|$ divides $|R^\times|$ and $\gcd(|R^\times|, p) = 1$). Hence, $|J| = 1$ and thus $J = \{0\}$. Now, $R$ must be an Artinian ring with trivial Jacobson radical, hence semisimple. By the Artin-Wedderburn Theorem, $R$ is a product of matrix rings over division rings, but since the division rings must be finite, they are in fact finite fields (of characteristic $p$) by Wedderburn's Little Theorem. This implies that \[ R \cong \prod_{i = 1}^n \mathbf{M}_{n_i}(\F_{p^{k_i}}).\]

Now suppose $n = n_i > 1$ for some $i$. Let $\mathbf{e}_{n-1} \in \R^{n - 1}$ be the coordinate vector $\mathbf{e}_{n-1} = (0, 0, \dots, 0, 1)$. One can check that the block matrix \[ X = \mm{1}{\mathbf{e}_{n-1}}{\mathbf{0}}{I_{n-1}} \] satisfies \[ X^k = \mm{1}{k\mathbf{e}_{n-1}}{\mathbf{0}}{I_{n-1}} \] and thus $X$ is a unit of order $p$. However, since $\gcd(|R^\times|, p) = 1$, $R^\times$ has no elements of order $p$, so $n_i = 1$ for all $i$ and the conclusion in the statement of the proposition now follows.
\end{proof}

The following construction is useful for building examples of unit groups.

\begin{proposition}\label{matrixtrick}
Let $R$ be a unital ring, let $H$ be an $R$-$R$-bimodule, and let \[ M(R, H) = \left\{ \mm{r}{h}{0}{r} \, | \, r \in R, h \in H \right\}. \] 
\begin{enumerate}
\item The set $M(R, H)$ is a unital ring and $R$-algebra under matrix addition and multiplication. \label{isring}
\item The rings $R$ and $M(R, H)$ have the same characteristic. \label{samechar}
\item The group of units $M(R, H)^\times$ is isomorphic to the group $(R^\times \times H, \cdot)$ where \[ (r, h)\cdot (s, k) = (rs, rk + hs) \text{ \ \  and \ \  }  (r, h)^{-1} = (r^{-1}, -r^{-1}hr^{-1}).\]\label{unitgroup}
\vspace{-.2in}
\item If $R$ is commutative and the bimodule structure on $H$ satisfies $rh = hr$ for all $r \in R$ and $h \in H$, then $M(R,H)^\times \cong R^\times \times (H, +)$.\label{communitgroup}
\item If $R$ is a local ring, then $M(R, H)$ is a local ring.\label{local}
\end{enumerate}
\end{proposition}
\begin{proof}
It is straightforward to check that $M(R, H)$ is a unital ring and an $R$-algebra under entry-wise addition and matrix multiplication; the identity is the usual identity matrix, and $R$ is a subring of $M(R,H)$ as the set of diagonal matrices. Since $M(R, H)$ is an $R$-algebra, the two rings have the same characteristic. This establishes (\ref{isring}) and (\ref{samechar}).

One can check that $A = \mm{r}{h}{0}{r}$ is invertible if and only if $r \in R^\times$, and \[ \mm{r}{h}{0}{r}^{-1} = \mm{r^{-1}}{-r^{-1}hr^{-1}}{0}{r^{-1}}.\] This establishes (\ref{unitgroup}). When $R$ is commutative, the map \[ \map{\Phi}{R^\times \times (H, +)}{(R^\times \times H, \cdot)} \] defined by $\Phi(r, k) = (r, rk)$ is a group isomorphism, establishing (\ref{communitgroup}).

Now suppose $R$ is a local ring with unique (two-sided) maximal ideal $J$. It is straightforward to check that \[ J_M = \left\{ \mm{j}{h}{0}{j} \, | \, j \in J, h \in H \right\}\] is the unique maximal ideal of $M(R, H)$; it is an ideal with $M(R, H)/J_M \cong R/J$ and its complement consists of units. Thus, $M(R, H)$ is a local ring, establishing (\ref{local}).
\end{proof}

\section{Rings of odd characteristic} \label{oddcharsec}

We will now determine the possible characteristics of a ring $R$ whose group of units is a $2$-group. First, observe that the ring may have characteristic zero, as the examples \[ \Z^\times = \C_2 \ \ \text{ and } \ \Z[i]^\times = \C_4 \] illustrate. For positive characteristics, we have the following theorem.

\begin{theorem} \label{2fchar}
Let $n$ be a positive integer. The following statements are equivalent.
\begin{enumerate}
\item  A $2$-group is realizable in characteristic $n$. \label{2gr}
\item The positive integer $n$ is of the form $2^ap_1p_2 \cdots p_k$ where the $p_i$'s are distinct Fermat primes, $a \ge 0$ and  $k \ge 0$. \label{fpp}
\item  It is possible to construct a regular $n$-gon with straightedge and compass. \label{ngon}
\end{enumerate}
\end{theorem}

\begin{proof}
The equivalence of (\ref{fpp}) and (\ref{ngon}) is classical (the Gauss-Wantzel Theorem). We show that (\ref{2gr}) and (\ref{fpp}) are equivalent. To see that (\ref{fpp}) implies (\ref{2gr}), let $n$ be as in (\ref{fpp}) and consider the ring $\Z_n$. Let $p_i = 2^{n_i}+1$.  Since the primes $2, p_1, \dots, p_k$ are pair-wise distinct,
\[ \Z_n = \Z_{2^a} \times \Z_{p_1} \times \cdots \Z_{p_k}.\] 
Taking units, we obtain 
\[ \Z_n^\times = \Z_{2^a}^\times \times \C_{2^{n_1}} \times \cdots  \C_{2^{n_k}}\]
Since $\Z_{2^a}^\times$ is always a 2-group, $\Z_{n}^\times$ is a 2-group. 

Finally, to see that (\ref{2gr}) implies (\ref{fpp}), let $R$ be a ring of characteristic $n > 0$ such that $R^\times = G$ is a $2$-group. Suppose $p$ is an odd prime with $p^t \mid n$ and $p^{t+1} \nmid n$. Then, $\Z_{p^t}$ is a factor ring of $\Z_n$. This means 
\[ \Z_{p^t}^\times  \subseteq \Z_n^\times \subseteq G. \]
It follows that $\Z_{p^t}^\times$ is a nontrivial $2$-group. Since $|\Z_{p^t}^\times| = p^{t-1}(p-1) $, we must have that $t=1$ and $p-1$ is a power of 2; i.e., $p$ is a Fermat prime.  This shows that $n$ has to be of the form given in (\ref{fpp}).\end{proof}

For a ring $R$ of characteristic $n$, where $n$ has the form given in Theorem \ref{2fchar} (\ref{fpp}), we have 
\[R  = S \times R_1 \times \cdots R_k  \]
where $S$ has characteristic $2^a$ and $R_i$ has characteristic $p_i$, a Fermat prime. This gives
\[R^\times  = S^\times \times R_1^\times \times \cdots R_k^\times.  \]
Which finite $2$-groups are realizable in each of these characteristics? The next proposition answers this question when the characteristic is a Fermat prime.

\begin{proposition}
Let $p$ be a Fermat prime and let $R$ be a ring of characteristic $p$ whose unit group is a finite $2$-group. Then,
\[ R^\times = \begin{cases}
(\C_2)^s \times (\C_8)^t  &   p = 3 \\
(\C_{2^n})^s &  p  = 2^n +1  > 3,
\end{cases}
\]
where the exponents $s$ and $t$ are non-negative integers. Conversely, the above groups are realizable in characteristic $p$ since $\F_9^\times \cong \C_8$ and, for $p = 2^n + 1$ a Fermat prime, $\F_p^\times \cong \C_{2^n}.$
\end{proposition}

\begin{proof}
Since the unit group of $R$ is finite, we may assume $R$ is finite by Proposition \ref{wlogRfinite}, and now by Proposition \ref{J0} we have that $R^\times = \prod_{i=1}^n \C_{p^{k_i} - 1}$. Hence, $p^{k_i} - 1$ is a power of $2$. This means either $p$ is $3$ and $k_i = 1$ or $2$, or $p$ is a Fermat prime bigger than $3$ and $k_i =1$ (see \cite[2.2]{cl-2} for a proof).
 \end{proof}

The peculiar occurrence of $\C_8$ when $p = 3$ is essentially a consequence of Catalan's Conjecture (proved in 2002 by  Preda Mih\u{a}ilescu), which says that the only nontrivial solution to $a^x - b^y = 1$ is $a = 3, x = 2, b = 2$, and $y = 3$. We now have the complete solution to Fuchs' problem for finite $2$-groups over rings of odd characteristic summarized in Theorem \ref{oddchar}.

\section{Rings of characteristic 2} \label{char2}

In this section we consider the question: which finite abelian $2$-groups are realizable in characteristic 2? This question is narrower in the obvious sense that we are restricting the characteristic of the ring; on the other hand, knowing that a $2$-group is realizable over the class of unital rings does not imply that it is realizable as the group of units of a ring of characteristic 2. For example, according to \cite{cl-3}, the only indecomposable $2$-groups that are realizable in characteristic 2 are $\C_2$ and $\C_4$; $\C_8$ is not, though of course $\C_8 \cong \F_9^\times$. 

Suppose $G$ is a finite abelian 2-group and $R$ is a ring of characteristic 2 with $R^\times \cong G$. There is a ring homomorphism \[\map{\phi}{\mathbf{F}_2[G]}{R}\] that restricts to the identity map on $G$. So $G$ is in fact the group of units of a quotient of the group algebra $\mathbf{F}_2[G]$. This group algebra is a tensor product of group algebras of the form $\F_2[\C_{2^n}] \cong \F_2[x]/(x^{2^n})$. The group of units in this ring is given below. 

\begin{theorem}[{\cite[2.4]{cl-4}}] The group of units of $\mathbf{F}_p[x]/(x^n)$ is isomorphic to \[\A_n = \C_{p-1} \times \left(\bigoplus_{1 \leq k < 1 + \log_p n} \C_{p^k}^{\left\lceil \frac{n}{p^{k-1}} \right\rceil - 2\left\lceil \frac{n}{p^k} \right\rceil + \left\lceil \frac{n}{p^{k+1}} \right\rceil}\right).\] In particular, \[ \A_{p^{r-1} + 1} = \mathbf{C}_{p-1} \times \mathbf{C}_{p^r} \times \mathbf{C}^{p-2}_{p^{r-1}} \times \mathbf{C}^{(p-1)^2}_{p^{r-2}} \times \mathbf{C}_{p^{r-3}}^{p(p-1)^2} \times \cdots \times  \mathbf{C}_{p^{2}}^{p^{r-4}(p-1)^2} \times  \mathbf{C}_{p}^{p^{r-3}(p-1)^2} \] and \[ \A_{p^r} = \C_{p-1} \times \C_{p^r}^{p-1}\times \C_{p^{r-1}}^{(p-1)^2} \times \C_{p^{r-2}}^{p(p-1)^2} \times \cdots \times \C_p^{p^{r-2}(p-1)^2}.\] If $p = 2$, then $\drank(\A_n) = \left\lfloor \frac{n}{2} \right\rfloor$.
\label{unitprop}
\end{theorem}
\begin{proof} The first statement is \cite[2.4]{cl-4}, and the next two equalities follow from that statement. In the proof of \cite[2.4]{cl-4}, it is shown that the number of elements in $\A_n$ of order a divisor of 2 is $\alpha_2 = 2^{n - \left\lceil \frac{n}{2} \right\rceil}$. Since $\A_n$ is an abelian $2$-group, this means that the canonical decomposition of $\A_n$ as a product of cyclic groups must have $n - \left\lceil \frac{n}{2} \right\rceil = \left\lfloor \frac{n}{2} \right\rfloor$ factors. This proves the final statement of the theorem.
\end{proof}

\begin{theorem}[{\cite[2.5]{cl-4}}] Let $R$ be a ring of prime characteristic $p$ whose group of units contains an element of order $p^r \geq p^2$. \label{rankprop}
\begin{enumerate}
\item If $p > 2$, then $R^\times$ has a noncyclic finite abelian subgroup $G$ such that $\trank(G) \geq 1 + (p-1)p^{r-2}$. \label{rankpropodd}
\item If $p = 2$, then $R^\times$ has a finite abelian subgroup $G$ such that $\trank(G) \geq 2^{r-2}$. If $r \geq 3$, then $G$ is noncyclic. \label{rankpropeven}
\end{enumerate}
\end{theorem}

Unfortunately, it is not easy to compute the group of units in a tensor product of rings (let alone a quotient of such a product). We have the following corollaries to the last two theorems.

\begin{corollary}
Given any finite abelian $2$-group $G$, there is a finite abelian $2$-group $H$ such that $G\times H$ is realizable as the group of units of a ring of characteristic $2.$
\end{corollary}
\begin{proof}
Every finite abelian $2$-group is a summand of $\A_n$ for some $n$.
\end{proof}

\begin{corollary} If $n$ is a positive integer and $G$ is a finite abelian $2$-group of rank $n$ that is realizable in characteristic $2$, then $|x| \leq 4n$ for all $x \in G$. Consequently, there are only finitely many abelian $2$-groups with $\drank(G) = n$ that are realizable in characteristic $2$. \label{rank-limitation}
\end{corollary}

\begin{corollary} If $G$ is a finite abelian $2$-group of order $2^m$ that is realizable in characteristic $2$, and if $x \in G$ has order $2^r$, then $2^{r-2}+r - 1 \leq m$.\end{corollary}
\begin{proof}
Take $G$ and $x$ as in the statement of the corollary. There must be one summand of $G$ of order at least $2^r$, and the total number of remaining summands is at most $m - r$. Hence, $m - r + 1 \geq \drank(G) \geq 2^{r-2}$ by Theorem \ref{rankprop}.
\end{proof}

\begin{ex}
The last corollary says, for example, that if you are interested in which groups of order at most $2^{11} = 2048$ are realizable, then you need only consider finite abelian $2$-groups with summands of size at most $16$.
\end{ex}

If $R$ is a ring with a finite group of units and $I$ is a two-sided ideal contained in the Jacobson radial of $R$, then the quotient map $R \longrightarrow R/I$ induces a surjective group homomorphism $R^\times \longrightarrow (R/I)^\times$ (see \cite[2.6]{cl-4}). This fact has the following nice consequence, since the $2^k$-power map is a ring homomorphism for commutative rings of characteristic 2. Let $G_{2^k} = \{ x^{2^k} \, | \, x \in G\}$.

\begin{proposition} \label{powertrick}
If $G$ is a finite abelian group that is realizable in characteristic $2$, then $G_{2^k}$ is realizable in characteristic $2$ for all $k.$
\end{proposition}
\begin{proof} Let $R$ be a ring of characteristic 2 with finite group of units $G$; without loss of generality, we may assume $R$ is commutative. The kernel $I$ of the map $x \mapsto x^{2^k}$ consists of nilpotent elements, so $I$ is contained in the Jacobson radical. Thus $(R/I)^\times \cong G_{2^k}$.
\end{proof}

If $G$ and $H$ are groups with $H \cong G_{2^k}$ for some non-negative integer $k$, then we will say that {\bf $G$ powers down to $H$}. The value of the above proposition is that it implies that any group that powers down to a group that is not realizable is itself not realizable.

\begin{corollary} \label{rankdifference}
If $H$ is a finite abelian $2$-group whose every element has order at most $2^r$, then $\C_{2^{r+3}} \times H$ is not realizable in characteristic $2$.
\end{corollary}
\begin{proof}
Any such group $G$ powers down to $\C_8$. Since $\C_8$ is not realizable in characteristic 2, neither is $G$, by Proposition \ref{powertrick}.
\end{proof}

The next two examples and Corollary \ref{rank-limitation} provide the proof of Theorem \ref{rank1-2}.

\begin{ex}[Rank 1 groups] \label{rank1} If $G$ is an indecomposable $2$-group, then it is realizable in characteristic 2 if and only if it is isomorphic to $\C_2$ or $\C_4$ (see \cite{cl-3}).\end{ex}

\begin{ex}[Rank 2 groups] \label{rank2} The rank 2 finite abelian $2$-groups that are realizable in characteristic 2 are $\C_8 \times \C_2$, $\C_4 \times \C_4$, $\C_4 \times \C_2$, and $\C_2 \times \C_2$. Suppose $G = \C_{2^a} \times \C_{2^b}$ with $b \leq a$. First, we know that $a \leq 3$ since the rank of $G$ is 2 and hence $8 \geq 2^a$ by Corollary \ref{rank-limitation}. By Corollary \ref{rankdifference}, we must have $a = b, a = b + 1,$ or $a = b + 2$. Thus, the possible realizable groups of rank 2 are $\C_8 \times \C_8$, $\C_8 \times \C_4$, $\C_8 \times \C_2$, $\C_4 \times \C_4$, $\C_4 \times \C_2$, and $\C_2 \times \C_2$. The last three groups are realizable by Example \ref{rank1}. We also have $\C_8 \times \C_2 = \A_5  \cong (\F_2[x]/(x^5))^\times.$ Neither $\C_8 \times \C_4$ nor $\C_8 \times \C_8$ is realizable by Proposition \ref{c48} below.
\end{ex}

\begin{proposition} \label{c48}
The groups $\C_8 \times \C_4$ and $\C_8 \times \C_8$ are each not the group of units of a ring of characteristic $2$.
\end{proposition}
\begin{proof}
Assume to the contrary that there is a ring of characteristic 2 whose group of units is isomorphic to $\C_8 \times \C_4$. Pick $g \in R^\times$ of order $8$. There is a ring homomorphism $\map{\phi}{\F_2[x]/(x^8)}{R}$ sending $1 + x$ to $g$. Since $\drank(\A_k) \geq 3$ for $k \geq 6$ and $\drank(R^\times) = 2$, we must have $x^5 \in \ker \phi$. So in fact there is a ring homomorphism $\map{\phi}{\F_2[x]/(x^5)}{R}$ sending $1 + x$ to $g$.  Further, we must have $x^4 \not \in \ker \phi$ for otherwise $g$ would have order dividing 4.

Now, $\A_5 \cong \cyclic{x+1} \times \cyclic{x^3 + 1} \cong \C_8 \times \C_2$. Since $\C_8$ is not realizable in characteristic 2, $1 + x^3$ must map to an element of order 2 in $R^\times$ that is not contained in the cyclic subgroup generated by $g$. Hence, $x^3 + 1$ maps to an element $h^2 \in R^\times$ where $h$ has order 4 and $R^\times \cong \cyclic{g} \times \cyclic{h}$. We may now extend $\phi$ to a homomorphism $\map{\phi}{\F_2[x, y]/(x^5, y^4)}{R}$ sending $1 + x$ to $g$ and $1+y$ to $h$. The kernel of $\phi$ must contain the relation $x^3 + 1 = y^2 + 1$ which is equivalent to $x^3 = y^2$.

Next, note that, modulo the kernel of $\phi$, $(1 + xy)^2 = 1 + x^2y^2 = 1 + x^5 = 1$. This means $1 + xy$ maps to a unit of order 1 or 2 (1, $g^4$, $h^2$, or $g^4 h^2$). This forces one of the following relations to hold modulo the kernel of $\phi$:
\[
\begin{aligned}
(1) \,\,\, xy & = 0 \implies x^4 = xx^3 = xy^2 = (xy)y = (0)y = 0 \\
(2) \,\,\, xy & = x^4 = xy^2 \implies xy(1 + y) = 0 \implies xy = 0 \implies x^4 = 0\\
(3) \,\,\, xy & = y^2 = x^3 \\
(4) \,\,\, xy & = 1 + (x^4 + 1)(y^2 + 1) \\
&= y^2 + x^4 + x^4y^2 \\
&= y^2 + xy^2 + x^4y^2 \\
&= y^2(1 + x + x^4) = x^3(1 + x + x^4) = x^3(1+x)
\end{aligned}
\]
\n The last two cases imply that $xy = x^3 u$ for some unit $u$ (modulo the kernel of $\phi$). Thus, in both cases, \[ x^4 = xy^2 = (xy)y = (x^3u)y = (xy)x^2u = (x^3u)x^2u = x^5u^2 = 0.\] Now we see that all four cases imply $x^4 = 0$ modulo the kernel of $\phi$, so no such ring $R$ exists.

The proof that $\C_8 \times \C_8$ is not realizable is similar. Assume to the contrary that there is a ring $R$ with $R^\times \cong \cyclic{g}\times \cyclic{h} \cong \C_8 \times \C_8.$ Then, there is a ring homomorphism $\map{\phi}{\F_2[x,y]/(x^5, y^5)}{R}$ sending $1+x$ to $g$ and $1 + y$ to $h$ such that $x^3 = y^4$ modulo the kernel of $\phi$, and $x^4 \not \in \ker \phi$. Further, the element $1 + xy^2$ maps to a unit of order 1 or 2. Considering four possible cases as above, one finds that this forces $x^4 = 0$ modulo the kernel of $\phi$, a contradiction.
\end{proof}

\begin{ex}[Groups that power down to $\C_8 \times \C_4$ or $\C_8 \times \C_8$] No such group is realizable by Proposition \ref{c48} and Proposition \ref{powertrick}. This eliminates some groups that are not eliminated by Theorem \ref{rankprop}; for example, $\C_{16} \times \C_8 \times \C_2^2$ and $\C_{16} \times \C_{16} \times \C_2^2$ are rank 4 groups which are not realizable in characteristic 2.
\end{ex}

\begin{question}[Rank 3 groups] Arguing as in Example \ref{rank2}, we may narrow down the list of possibilities here. Removing groups we know are realizable as direct products of previous cases, we have the following list of rank 3 groups with unknown status: $\C_8 \times \C_8 \times \C_8, \C_8 \times \C_8 \times \C_4, \C_8 \times \C_8 \times \C_2, \C_8 \times \C_4 \times \C_4.$ Which, if any, of these groups are realizable in characteristic 2? The latter two groups are the only finite abelian $2$-groups of at most order 128 with unknown status.
\end{question}

Regarding nonabelian $2$-groups, we conclude with two examples that allow one to generate realizable nonabelian $2$-groups over rings of characteristic 2.

\begin{ex}[Units in $\U_n(\mathbf{F}_2)$]
Consider the ring $\U_n(\F_2)$ of upper triangular matrices over the field of two elements $\F_2$. A matrix in this ring (which has characteristic $2$ and is non-commutative when $n > 2$)  is invertible if and only if all of its diagonal entries are 1. So the group of units is nonabelian and has $2^{(1+2+3+...+n-1)} = 2^{(n(n-1)/2)}$ elements.
\end{ex}

\begin{ex}[Units in group algebras] We begin with a proposition.

\begin{proposition}
 For any prime $p$ and any finite $p$-group $G$, in the group ring $\mathbf{F}_pG$ we have $|(\mathbf{F}_pG)^\times| = (p-1) p^{(|G|-1)}$.
\end{proposition}
\begin{proof}  It is well known that $\mathbf{F}_pG$ is a local ring (the augmentation ideal is the unique maximal ideal). Therefore every unit in $\mathbf{F}_pG$ has a unique (because we are working over a field) representation of the form  $u + x$  where $u$ is a unit in $\mathbf{F}_p$ and $x$ is an element in the augmentation ideal of $\mathbf{F}_pG$.   There are $p-1$ choices for $u$. 
Since the sum of the coefficients $t_g$  of  any element $\sum_g  t_g g$ in the augmentation ideal has to be zero,  the first $|G|-1$ coefficients can be filled with any  element of $\mathbf{F}_p$ and then the last coefficient will be determined.  So the number of choices for $x$ is $p^{(|G|-1)}$. This means the total number of units in $\mathbf{F}_pG$ is  $(p-1)p^{(|G|-1)}$
\end{proof}

In particular, for any finite 2-group $G$, $(\F_2G)^\times$ is a 2-group of order $2^{|G|-1}$ containing $G$ as a subgroup. The smallest nonabelian $2$-groups are $\Dih_8$ and $\Quat_8$. The unit groups of $\mathbf{F}_2\Dih_8$ and $\mathbf{F}_2\Quat_8$ are nonabelian (because they contain $\Dih_8$ and $\Quat_8$ as subgroups) $2$-groups of order $2^7$.
\end{ex}

\section{Rings of characteristic $2^n\, (n > 1)$} \label{char2n}

In this section, we briefly report some of what is known about the realizability of abelian $2$-groups in characteristic $2^n$ (for $n > 1$). An obvious ring of characteristic $2^n$ to consider is $\Z_{2^n}$. Using this ring and its modules, we may easily construct rings of characteristic $2^n$ with unit groups that are finite abelian $2$-groups using Proposition \ref{matrixtrick}. Recall that the exponent $\mathrm{exp}(G)$ of a finite group $G$ is the least common multiple of the orders of its elements. The realizability of the unit groups below is already known; for example, they are precisely the $2$-groups entailed by \cite[4.8]{ddx}.
\begin{proposition} \label{2ngroups}
Let $Q$ be a finite abelian $2$-group of exponent $2^m$. For any integer $n \geq \max\{m, 2\}$, the group \[ \C_2 \times \C_{2^{n-2}} \times Q\] is realizable in characteristic $2^{n}$.
\end{proposition}

\begin{proof}
Since $n \geq m$, we have that $Q$ is an $\Z_{2^n}$-module. By Proposition \ref{matrixtrick}, \[ M(\Z_{2^n}, Q)^\times = \Z_{2^n}^\times \times Q = \C_2 \times \C_{2^{n-2}} \times Q\] and $\mathrm{char}\, M(\Z_{2^n}, Q) = \mathrm{char}\, \Z_{2^n} = 2^n$.
\end{proof}

If one is willing to give up the requirement that the realizing ring be local (e.g., if one is only interested in the general question of whether a group is the group of units in a finite ring of characteristic a power of 2), then we may use the above method to recover all the unit groups in \cite[4.8]{ddx}. To see why, let $P$ be a finite abelian $2$-group of exponent $2^a$, suppose $a_0 \geq a - 1$, and let $Q = (\C_{2^{a_0}})^{\lambda - 1} \times P^\lambda$. Now, $Q$ is a $\Z_{a_0 + 1}$-module and
\[
\begin{aligned} (\F_{2^\lambda} \times M(\Z_{2^{a_0+1}}, Q))^\times & \cong \F_{2^\lambda}^\times \times M(\Z_{2^{a_0+1}}, Q)^\times \\ &= \F_{2^\lambda}^\times \times \C_2 \times \C_{2^{a_0-1}} \times (\C_{2^{a_0}})^{\lambda - 1} \times P^\lambda.\end{aligned}\] Note that $\F_{2^\lambda} \times M(\Z_{2^{a_0+1}}, Q)$ has characteristic $2^{a_0 + 1}$. Further, the unit group in \cite[\S 5 Example 3]{ddx} is the group of units in the ring \[ \F_4 \times M(\Z_{2^n}, \C_2^2 \times \C_4 \times \C_{2^{n-1}})\] of characteristic $2^n$ (where $n \geq 2$).

Another family of examples may be obtained by taking certain quotients of the Gaussian integers. The quotient ring $\Z[i]/(1+i)^n$ has characteristic $2^n$ and, for $n \geq 5$,
\[ (\Z[i]/(1+i)^n)^\times \cong \begin{cases} \C_{2^{m-1}} \times \C_{2^{m-2}} \times \C_4& \text{ if $n = 2m$}\\ \C_{2^{m-1}} \times \C_{2^{m-1}} \times \C_4 & \text{ if $n = 2m+1.$}\end{cases}\]
For $n = 1, 2, 3, 4$ one obtains $\C_1, \C_2, \C_4, \C_2 \times \C_4$. These unit groups are of course well known; see, for example, \cite{cross}.


\begin{remark} In this remark, we explain how the results \cite[4.3, 4.8]{ddx} fit neatly into the framework for generating unit groups given by Proposition \ref{matrixtrick}. A {\bf Galois ring of characteristic $p^n$} is a ring of the form \[ R = \Z_{p^n}[t]/(f(t)),\] where $f(t)$ is a monic polynomial with integer coefficients whose reduction modulo $p$ is irreducible of degree $\lambda$. Galois rings are generalizations of finite fields and have been used in the study of finite rings since the late 1960s. Their unit groups are well known; they were computed in \cite[Theorem 9]{rag}: for $n \geq 2$, \[
R^\times \cong \begin{cases} \F_{2^\lambda}^\times \times \C_2 \times \C_{2^{n-2}} \times (\C_{2^{n-1}})^{\lambda - 1} & \text{ if $p = 2$} \\ \F_{p^\lambda}^\times \times \C_{p^{n-1}}^\lambda & \text{ if $p > 2$.} \end{cases}\] 
For $1 \leq k \leq n$, $R/(p^k)$ is an $R$-module that is isomorphic to $\C_{p^k}^\lambda$ as an abelian group. Hence, if $G$ is a finite abelian $p$-group of exponent at most $2^n$, then $G^\lambda$ is an $R$-module. Proposition \ref{matrixtrick} now implies the following, for any $\lambda \geq 1$:
\begin{enumerate}
\item If $p = 2$, $T$ is a finite abelian $2$-group of exponent $2^m$, and $n \geq m$, then \[ M(R, T^\lambda)^\times \cong \F_{2^\lambda}^\times \times \C_2 \times \C_{2^{n-2}} \times (\C_{2^{n-1}})^{\lambda - 1} \times T^\lambda.\] \label{ddxn1}
\vspace{-.2in}
\item If $p > 2$, $D$ is a finite abelian $p$-group of exponent $2^m$, and $n \geq m$, then \[ M(R, D^\lambda)^\times \cong \F_{p^\lambda}^\times \times \C_{p^{n-1}}^\lambda \times D^\lambda.\] \label{ddxn2}
\end{enumerate}
\vspace{-.2in}
Item (\ref{ddxn1}) is \cite[4.8]{ddx}, and item (\ref{ddxn2}) is equivalent to \cite[4.3]{ddx}: given a finite abelian $p$-group $P$ of exponent $p^m$, we may write $P = \C_{p^{m}} \times D$ for some finite abelian $p$-group $D$ with $\mathrm{exp}(D) \leq p^m$. Now take $n = m +1$ above, and $M(R, D^\lambda)^\times \cong \F_{p^\lambda}^\times \times \C_{p^{m}}^\lambda \times D^\lambda \cong \F_{p^\lambda}^\times \times P^\lambda.$
\end{remark}

\section{Rings of characteristic 0} \label{char0sec}

In this section, we summarize what we know about the realizability of finite abelian $2$-groups in characteristic zero. As mentioned in \S \ref{oddcharsec}, we have $\Z^\times \cong \C_2$ and $\Z[i]^\times \cong \C_4$. This means we may use Proposition \ref{matrixtrick} to generate examples with either $\C_2$ or $\C_4$ as a summand.

\begin{ex}[Groups of type $\C_2 \times P$] \label{mt1} Using Proposition \ref{matrixtrick}, we already know that $\C_2 \times P$ is the group of units of a ring of characteristic zero for any abelian $2$-group $P$.
\end{ex}

Recall from the introduction that, given a finite abelian $2$-group $B$, $B_2 = \{x^2 \, | \, x \in B\}$.

\begin{ex}[Groups of type $\C_4 \times P$] \label{mt2} Using Proposition \ref{matrixtrick}, coupled with the fact that $\Z[i]^\times \cong \C_4$, we may realize any abelian $2$-group of the form $\C_4 \times P$ when $P$ is a $\Z[i]$-module which happens to be abelian $2$-group. Since the Gaussian integers are a PID, all finitely generated modules are sums of modules of the form $\Z[i]/(a^k)$ where $a$ is a Gaussian prime. The only such quotients that are $2$-groups under addition are the modules $\Z[i]/(1+i)^k$, isomorphic to $\C_{2^m} \times \C_{2^m}$ if $k = 2m$ and isomorphic to $\C_{2^{m+1}} \times \C_{2^m}$ if $k = 2m + 1$. Thus, $P$ may be any direct sum of copies of groups of the form $\C_{2^m} \times \C_{2^m}$ and $\C_{2^{m+1}} \times \C_{2^m}$, where $m$ may be any positive integer. Grouping factors, we obtain that $\C_4 \times A \times A \times B \times B_2$ is realizable in characteristic zero for any pair of finite abelian $2$-groups $A$ and $B$. These examples also appear in \cite{dd} (all of the realizable groups in Proposition 5.8 of that paper may also be obtained using Proposition \ref{matrixtrick} and the well-known description of the modules $\Z[i]/(a^k)$ for $a$ a Gaussian prime).
\end{ex}

\begin{question}[The group $\C_4 \times \C_{32}$] So far as the authors are aware, it is unknown whether the group $\C_4 \times \C_{32}$ is the group of units in a ring. (This is the smallest abelian $2$-group for which the question is open.) By Example \ref{rank2}, it is not realizable in characteristic 2, and the group of units in a ring of characteristic $2^n$ must contain a subgroup isomorphic to $\C_2 \times \C_{2^{n-2}}$ when $n \geq 2$. We do not know whether it is the group of units in a ring of characteristic 0 or $2^n$ with $2 < 2^n \leq 2^7$.
\end{question}

We next address $2$-groups of the form $\C_8 \times P$. First, we will state and prove a proposition that has been extracted from the proof of \cite[Proposition 2.2]{cl-3}.

\begin{proposition} Let $R$ be a ring whose group of units is a $2$-group.  If $-1 \in R$ has a fourth root, then $R$ has positive characteristic.
\end{proposition}
\begin{proof}
Let $R$ be a ring and let $G = R^\times$ be a $2$-group. Suppose there exists $z \in R$ such that $z^4 + 1 = 0$. This of course implies that $z$ is a unit. The element $z^2 + z + 1$ is also a unit (its inverse is $1 - z^2 + z^3$), so since $G$ is a $2$-group, we have $(z^2 + z + 1)^{2^k} - 1 = 0$ for some positive integer $k$. There is a ring homomorphism $\map{\phi}{\Z[x]}{R}$ sending $x$ to $z$, and the ideal $I = (x^4 + 1, (x^2 + x + 1)^{2^k} - 1)$ is in the kernel of $\phi$. If we can prove that this ideal contains a nonzero integer, then $R$ will have positive characteristic. It is straightforward to show (see \cite[2.2]{cl-3} for the details) that the two given generators of $I$ are relatively prime in $\Q[x]$, and this implies that $I$ contains a nonzero integer. This completes the proof.
\end{proof}

\begin{corollary}
If $G$ is a finite $2$-group that does not have $\C_2$ or $\C_4$ as a summand, then $G$ is not realizable in characteristic zero.
\end{corollary}
\begin{proof}
Take $G$ as in the statement of the corollary and assume to the contrary that there is a ring $R$ of characteristic 0 with $R^\times \cong G$. The element $-1 \in R$ is a unit of order 2, and since $\C_2$ and $\C_4$ are not summands of $G$, we must have that there exists an element $z \in G$ with $z^4 + 1 = 0$. This contradicts the conclusion of the proposition above.
\end{proof}

The above results prove Theorem \ref{char0main}. For abelian 2-groups over rings of characteristic 0, it therefore remains to determine which finite abelian $2$-groups of the form $\C_4 \times P$ (where $P$ does not have $\C_2$ as a summand) are realizable.

\section{Nonabelian $2$-groups} \label{nonabelian}

In this section we will consider nonabelian $2$-groups. We will prove a proposition which will give restrictions on the center and cyclic maximal abelian subgroups of a realizable nonabelian $2$-group. We will use this result  to solve Fuchs' problem for almost cyclic $p$-groups  and groups with periodic cohomology.

\begin{proposition} \label{main-nonabelian}
Let $G$ be a finite realizable nonabelian $2$-group. Then:
\begin{enumerate}
\item If $Z(G)$ is cyclic, then $Z(G)$ is isomorphic to $\C_2$ or $\C_4$. \label{main-nonabelian-z}
\item Each cyclic maximal abelian subgroup of $G$ is isomorphic to $\C_2$ or $\C_4$. \label{main-nonabelian-cm}
\end{enumerate}
\end{proposition}

\begin{proof} We begin with the observation that if the center of a $2$-group $G$ is cyclic, then $G$ is indecomposable; this follows from the fact that nontrivial $p$-groups have nontrivial centers. Since the center sits inside 
every maximal abelian subgroup, we may conclude that if a $2$-group $G$ has a cyclic maximal abelian subgroup, then $G$ is indecomposable.

We will prove (A) and (B) simultaneously. Let $M$ be either the center or a cyclic maximal abelian subgroup of $G$. Note that $M \ne G$ since $G$ is nonabelian. 
Moreover, by the observation made above,   $G$ has to be indecomposable. Let $R$ be a finite ring such that $R^\times = G$. 
Since $G$ is indecomposable,  Theorem  \ref{2fchar} implies that  the characteristic of $R$ has to be either $0$, $2^k$ for some $k$, or a Fermat prime.  The Fermat prime case is ruled out by Proposition \ref{J0} since $G$ is a nonabelian indecomposable group. By Propositions \ref{maximal} and \ref{center},  the cyclic $2$-group $M$ is also realizable in characteristic $0$ or $2^k$. By \cite{cl-3}, we then have that $M$ is isomorphic to $\C_2$ or $\C_4$ (these are the only cyclic $2$-groups which are realizable in characteristics 0 or $2^k$).\end{proof}

\subsection{Almost cyclic $2$-groups} \label{ac2g}
We determine all realizable almost cyclic $2$-groups; i.e., $2$-groups which have a cyclic subgroup of index $2$.  We begin with the abelian case. First, observe that any abelian almost cyclic $2$ group must be either $\C_{2^r}$ or $\C_{2^r} \times \C_2$ for some $r$. (This can be seen directly or from the classification of finite abelian groups.)  The indecomposable case is handled in \cite{cl-3}.

\begin{theorem}[\cite{cl-3}] \label{abelian-1}
The group $\C_{2^r}$ is realizable if and only if\, $\C_{2^r}$ is $\C_2$, $\C_4$, $\C_8$, $\C_{q-1}$ where $q$ is a Fermat prime.
\end{theorem}

\n It was also proved in \cite{cl-3} that the only realizable indecomposable abelian $p$-groups for $p$ an odd prime are the groups $\C_p$ for $p$ Mersenne. For finite abelian $p$-groups, this fact also follows from Ditor's work.

Now we consider the the groups $\C_{2^r} \times \C_2$. The following proposition shows that all these groups are realizable.

\begin{proposition}\label{abelian-2}
If $H$ is an abelian group, then the group $\C_2 \times H$ is realizable. 
\end{proposition}
\begin{proof}
An abelian group $H$ is a $\Z$-module and $\Z^\times \cong \C_2$. Now apply Proposition \ref{matrixtrick}.
\end{proof}

Before turning our attention to nonabelian almost cyclic $2$-groups, we first remind the reader that, in the odd-primary case, there are no realizable nonabelian $p$-groups. This follows from the work of Ditor; however, we include a short proof here for convenience.

\begin{proposition}
If $p$ is odd, then every realizable finite $p$-group is abelian.
\end{proposition}
\begin{proof}
Let $G$ be a finite $p$-group that is realizable by a ring $R$.  Without loss of generality, we can assume that $R$ is finite because $G$  is finite.  The ring $R$ must have characteristic $2$, for otherwise $-1$ would be a unit of order 2 which is impossible when $p$ is odd. We now have $\gcd(|R^\times|, \mathrm{char}(R)) = 1$, so Proposition \ref{J0} implies that $R^\times =  \prod \F_{2^{k_i}}^\times$ is an abelian group.
\end{proof}

We now consider nonabelian almost cyclic $2$-groups of order $2^n$. The only nonabelian groups of order 8, $\Dih_8$ and $\Quat_8$, are both almost cyclic.  It is well known that $\Dih_8$ is realizable in characteristic 2; for example, $(\U_3(\F_2))^\times = \Dih_8$, where $\U_3(\F_2)$ is the ring of upper-triangular $3 \times 3$ matrices with entries in $\F_2$. It is also well known that $\Quat_8$  is realizable. Consider the ring of Lipschitz integers defined by 
\[ L = \{ a + bi + cj + dk  \, | \, a, b, c, d \in \Z \} \]
where $i, j$ and $k$ satisfy the relations $i^2 = j^2 = k^2 = -1$ and $ij = k, jk =i,  ki =j$.  It can be shown that a Lipschitz integer $a + bi + cj + dk$ is a unit if and only if its norm $a^2 + b^2 +c^2 +d^2 = 1$. This shows that that the units in this ring are $\{ \pm 1, \pm i, \pm j, \pm k \}$ with the above relations. This group is isomorphic to $\Quat_8$. Combining this with what we already know about the realizability of groups of odd order, we have the following proposition.

\begin{proposition}
The following is a complete list of realizable groups of order $p^3$ for some prime $p$: $\C_4 \times \C_2$, $\C_8$, $\Dih_8$, $\Quat_8$, and $\C_p^3$ for $p = 2$ or $p$ a Mersenne prime.
\end{proposition}

It remains to consider nonabelian almost cyclic $2$-groups of order at least $16$. All these groups have a cyclic maximal abelian subgroup of order at least $8$, so none of these groups are realizable according to Proposition \ref{main-nonabelian} (\ref{main-nonabelian-cm}).  This completes the proof of Theorem \ref{maintheorem} (\ref{maintheorem-ac}). Note further that this implies that the converse of Theorem \ref{main-nonabelian} (\ref{main-nonabelian-z}) is not true: for example, $\Quat_{16}$ is a finite nonabelian 2-group with $Z(\Quat_{16}) \cong \C_2$, but $\Quat_{16}$ is not realizable.

\subsection{Groups with periodic cohomology.}

A finite group $G$ is said to have periodic mod-$p$ cohomology if there is a cohomology class $\eta$ in $H^d(G, \mathbb{F}_p)$ such that multiplication by $\eta$ gives an isomorphism (for all $i > 0$)
\[ H^i(G, \mathbb{F}_p) \cong H^{i+d}(G, \mathbb{F}_p). \]
It is well known that a finite group $G$ has periodic mod-$p$ cohomology if and only the Sylow $p$-subgroup of $G$ is either a cyclic group or a generalized quaternion group. Since we now have a solution to Fuchs' problem for both cyclic and generalized quaternion groups, we have proved Theorem \ref{maintheorem} (\ref{maintheorem-pc}).

\end{document}